\documentclass[11pt,twoside,a4paper]{article}

\usepackage{graphicx,url}
\usepackage{setspace}
\usepackage{enumerate,lineno,setspace,float}
\usepackage[small,compact]{titlesec}
\usepackage{amsmath,amsfonts,amssymb,amsthm}
\usepackage{geometry}
\usepackage{fancyhdr}
\usepackage{amssymb,amsmath}
\usepackage{latexsym}
\usepackage{graphicx, float}
\usepackage{url}

\newtheorem{Theorem}{Theorem}[section]

 \newtheorem{Corollary}[Theorem]{Corollary}
\newtheorem{Lemma}[Theorem]{Lemma}
\newtheorem{Example}[Theorem]{Example}

\newtheorem{Remark}[Theorem]{Remark}

\long\def\delete#1{}

\input amssym.def
\input amssym.tex

\newcommand{\be}{\begin{equation}}
\newcommand{\ee}{\end{equation}}
\newcommand{\bea}{\begin{eqnarray}}
\newcommand{\eea}{\end{eqnarray}}
\newcommand{\bean}{\begin{eqnarray*}}
\newcommand{\eean}{\end{eqnarray*}}

\def\diam{{\rm diam}}
\def\span{{\rm span}}
\def\hc{{\rm hc}}

\def\ve{\varepsilon}

\def\L{\mathcal{L}}

\def\({\left(}
\def\){\right)}
\def\[{\left[}
\def\]{\right]}

\begin{document}

\title{Hamiltonian chromatic number of trees}
\author{\textbf{Devsi Bantva} \\ Department of Mathematics \\ Lukhdhirji Engineering College, Morvi - 363 642, Gujarat (INDIA) \\ E-mail : \textit{devsi.bantva@gmail.com} \\ \\ \textbf{Samir Vaidya} \\ Department of Mathematics \\ Saurashtra University, Rajkot - 360 005, Gujarat (INDIA) \\ E-mail : \textit{samirkvaidya@yahoo.co.in}}

\pagestyle{myheadings}
\markboth{\centerline{Devsi Bantva and Samir Vaidya}}{\centerline{Hamiltonian chromatic number of trees}}
\date{}
\openup 0.8\jot
\maketitle

\begin{abstract}
Let $G$ be a simple finite connected graph of order $n$. The detour distance between two distinct vertices $u$ and $v$ denoted by $D(u,v)$ is the length of a longest $uv$-path in $G$. A hamiltonian coloring $h$ of a graph $G$ of order $n$ is a mapping $h : V(G) \rightarrow \{0,1,2,...\}$ such that $D(u,v) + |h(u)-h(v)| \geq n-1$, for every two distinct vertices $u$ and $v$ of $G$. The span of $h$, denoted by $\span(h)$, is $\max\{|h(u)-h(v)| : u, v \in V(G)\}$. The \emph{hamiltonian chromatic number} of $G$ is defined as $\hc(G) := \mbox{min}\{\span(h)\}$ with minimum taken  over all hamiltonian coloring $h$ of $G$. In this paper, we give an improved lower bound for the hamiltonian chromatic number of trees and give a necessary and sufficient condition to achieve the improved lower bound. Using this result, we determine the hamiltonian chromatic number of two families of trees.
\end{abstract}

\section{Introduction}\label{sec:intro}

Let $G$ be a simple finite connected graph with vertex set $V(G)$ and edge set $E(G)$. The \emph{order} of a graph $G$ written as $|G|$ is the number of vertices in $G$. For a vertex $v \in V(G)$, the \emph{neighborhood} of $v$ denoted by $N(v)$ is the set of vertices adjacent to $v$. The \emph{distance $d(u,v)$} between two vertices $u$ and $v$ is the length of a shortest path joining $u$ and $v$. The \emph{detour distance} between two vertices $u$ and $v$ denoted by $D(u,v)$ is the length of a longest path joining $u$ and $v$ (refer \cite{Chartrand3} for more details on it). The \emph{diameter} of a graph $G$ denoted by $\diam(G)$ or simply $d$ is $\max\{d(u,v) : u, v \in V(G)\}$. The \emph{eccentricity} $\epsilon(v)$ of a vertex $v \in V(G)$ is the distance from $v$ to a vertex farthest from $v$. The \emph{center} $C(G)$ of graph $G$ is the subgraph of $G$ induced by the vertex/vartices of $G$ whose eccentricity is minimum. Moreover, for standard graph theoretic terminology and notation we follow \cite{West}.

A \emph{hamiltonian coloring} $h$ of a graph $G$, introduced by Chartrand \emph{et al.} in \cite{Chartrand1}, is a mapping $h : V(G) \rightarrow \{0, 1, 2,...\}$ such that for every pair of distinct vertices $u,v$ of $G$,
\be\label{hc:dfn} D(u,v)+|h(u)-h(v)| \geq n-1. \ee
The \emph{span of $h$}, denoted by $\span(h)$, is defined as max$\{|h(u)-h(v)| : u, v \in V(G)\}$. The \emph{hamiltonian chromatic number} $\hc(G)$ of $G$ is
\bean \hc(G) := \mbox{min}\{\span(h)\} \eean
with minimum taken  over all hamiltonian coloring $h$ of $G$. A hamiltonian coloring $h$ of $G$ is called \emph{optimal} if $\span(h)$ = $\hc(G)$.

It is clear from definition that if $G$ contains a hamiltonian $uv$-path between two distinct vertices $u$ and $v$ then the same color can be assigned to both $u$ and $v$. Hence, a graph $G$ is hamiltonian-connected if and only if $G$ can be hamiltonian colored by a single color. In \cite{Chartrand1}, Chartrand \emph{et al.} proved that for any two integers $j$ and $n$ with $2 \leq j \leq (n+1)/2$ and $n \geq 6$, there is a hamiltonian graph of order $n$ with hamiltonian chromatic number $n-j$. Thus the hamiltonian chromatic number of a connected graph $G$ measures how close $G$ is to being hamiltonian-connected. Without loss of generality we allow 0 as a color in the definition of the hamiltonian coloring, then the \emph{span} of any hamiltonian coloring $h$ is the maximum integer used for coloring while in \cite{Chartrand1,Chartrand2,Shen} only positive integers are used as colors. Therefore, the hamiltonian chromatic number defined in this article is one less than that defined in \cite{Chartrand1,Chartrand2,Shen} and hence we will make necessary adjustment when we present the results of \cite{Chartrand1,Chartrand2,Shen} in this article.

The hamiltonian chromatic number was introduced by Chartrand \emph{et al.} in \cite{Chartrand1} as a variation of \emph{radio antipodal coloring} of graphs but it is less explored compare to it. A very few results presented by researchers for hamiltonian chromatic number of graphs. The hamiltonian chromatic number of some well-known graph families is determined by Chartrand \emph{et al.} in \cite{Chartrand1} which are as follows: $\hc(K_{n}) = 0$ for $n \geq 1$, $\hc(C_{n}) = n-3$ for $n \geq 3$, $\hc(K_{1,n-1}) = (n-2)^{2}$ for $n \geq 3$, $\hc(K_{r,r}) = r-1$ for $r \geq 1$, $\hc(K_{r,s}) = (s-1)^{2}-(r-1)^{2}-1$ for $2 \leq r <s$ and gave an upper bound for $\hc(P_{n})$. However, it is noted that $hc(P_{n})$ is same as radio antipodal number $ac(P_{n})$ which is determined in \cite{Khennoufa}. They proved that if $T$ is a spanning tree of a connected graph $G$ then $\hc(G) \leq \hc(T)$ and for any tree $T$ of order $n \geq 2$, $\hc(T) \leq (n-2)^{2}$. They also proved that for any two integers $j$ and $n$ with $2 \leq j \leq (n+1)/2$ and $n \geq 6$, there is a hamiltonian graph of order $n$ with hamiltonian chromatic number $n-j$. In \cite{Chartrand2}, the same group of authors gave a lower bound for circumference of $G$ (the circumference of a graph $G$ is the length of a longest cycle in $G$) is given in terms of the number of vertices that receive colors between two specified colors in a hamiltonian coloring of $G$. The authors also proved that for a connected graph $G$ of order $n \geq 3$, if $(n+2)/2$ vertices receive the same color in a hamiltonian coloring then $G$ is hamiltonian. In \cite{Shen}, Shen \emph{et al.} determined the hamiltonian chromatic number of graph $G$ with $\max\{D(u,v) : u, v \in V(G), u \neq v\} \leq n/2$ and illustrated the result with a special class of caterpillars and double stars.

This paper is organized as follows. In section \ref{sec:pre}, we define all necessary terms, notations and terminologies for present work. In section \ref{sec:lb}, we give improved lower bound for the hamiltonian chromatic number of trees and present a necessary and sufficient condition to achieve the improved lower bound. We determine the hamiltonian chromatic number of two families of trees in section \ref{sec:vs}. In this section, we show that for a special type of broom trees (see section \ref{subsec:bk} for definition and detail on it) the lower bound given in present work is better than one given in \cite[Theorem 4]{Bantva1}. We explain the strength of our results in concluding remarks section.

\section{Preliminaries}\label{sec:pre}

A tree $T$ is a connected graph that contains no cycle. In \cite{Liu} and later used in \cite{Bantva}, the \emph{weight of $T$} from $v \in V(T)$ is defined as $w_{T}(v)$ = $\sum_{u \in V(T)}d(u,v)$ and the \emph{weight of $T$} as $w(T)$ = min$\{w_{T}(v) : v \in V(T)\}$. A vertex $v \in V(T)$ is a \emph{weight center} of $T$ if $w_{T}(v)$ = $w(T)$. Denote the set of wight centers of $T$ by $W(T)$. It was proved in \cite{Liu} that every tree $T$ has either one or two weight centers, and $T$ has two weight centers, say $W(T) = \{w,w'\}$, if and only if $w$ and $w'$ are adjacent and $T-\{ww'\}$ consists of two equal-sized components. We set $W(T) = \{w\}$ as a root if $T$ has only one weight center $w$ and $W(T) = \{w,w'\}$ as a root if $T$ has two adjacent weight centers $w$ and $w'$. In either case, if a vertex $u$ is on the path joining weight center and a vertex $v \in V(T)$ then $u$ is called \emph{ancestor} of $v$ and $v$ is called \emph{descendent} of $u$. If $u$ is ancestor of $v$ which is adjacent to $v$ then $u$ is called \emph{parent} of $v$ and $v$ is called a \emph{child} of $u$. The subtree induced by a child $u$ of a weight center and all descendent of $u$ is called \emph{branch} at $u$. Two branches are called \emph{different} if they are at two vertices adjacent to the same weight center, and \emph{opposite} if they are at two vertices adjacent to different weight centers. Note that the concept of opposite branches occurs only when $T$ has two weight centers.

Define
$$
\L(u) := \min\{D(u,w) : w \in W(T)\}, u \in V(T)
$$
to indicate the \emph{detour level} of $u$ in $T$. Define the \emph{total detour level of $T$} as
$$
\L_{W}(T) := \displaystyle\sum_{u \in V(T)} \L(u).
$$

For any $u, v \in V(T)$, define
\bean\label{def:phi}
\phi(u,v) := \mbox{max}\{\L(t) : t \mbox{ is a common ancestor of $u$ and $v$}\},
\eean
\begin{eqnarray*}\label{def:delta}
\delta(u,v) := \left\{
  \begin{array}{ll}
    1, & \hbox{If $W(T) = \{w,w'\}$ and $uv$-path contains the edge $ww'$}, \\
    0, & \hbox{otherwise.}
  \end{array}
\right.
\end{eqnarray*}

\begin{Lemma}\label{lem:phi}~Let $T$ be a tree with diameter $d \geq 2$. Then for any $u, v \in V(T)$, the following hold:
\begin{enumerate}[\rm (a)]
  \item $\phi(u,v) \geq 0$;
  \item $\phi(u,v)$ = 0 if and only if $u$ and $v$ are in different or opposite branches;
  \item $\delta(u,v)$ = 1 if and only if $T$ has two weight centers and $u$ and $v$ are in opposite branches.
  \item the detour distance $D(u,v)$ in $T$ between $u$ and $v$ can be expressed as \be\label{eq:duv} D(u,v) = \L(u) + \L(v) -2\phi(u,v)+
  \delta(u,v).\ee
\end{enumerate}
\end{Lemma}
Note that the detour distance $D(u,v)$ is same as the ordinary distance $d(u,v)$ between any two vertices $u$ and $v$ in a tree $T$ but we continue to use this terminology to remain consistent with concept of hamiltonian coloring in general.

Define
\begin{eqnarray*}
\zeta(T) := \left\{
  \begin{array}{ll}
    0, & \hbox{If $W(T) = \{w\}$,} \\
    1, & \hbox{If $W(T) = \{w,w'\}$},
  \end{array}
\right.
\end{eqnarray*}

and
$$\zeta'(T) := 1-\zeta(T).$$

\section{A lower bound for $\hc(T)$}\label{sec:lb}

In this section, we continue to use terms, notations and terminologies defined in previous section.

A hamiltonian coloring $h$ of a tree $T$ is injective if $T$ has at least one vertex of degree 3, that is, $T$ is not a path as in this case no two vertices of trees contain hamiltonian path. So we assume all trees with at least one vertex of degree 3 through out this discussion unless otherwise specified. Observe that a hamiltonian coloring $h$ on $V(T)$, induces an ordering of $V(T)$, which is a line-up of the vertices with increasing images. We denote this ordering by $V(G) = \{x_{0},x_{1},x_{2},...,x_{n-1}\}$ with
$$
0 = h(x_{0}) < h(x_{1}) < ... < h(x_{n-1}) = \span(h).
$$

The next result gives a lower bound for the hamiltonian chromatic number of trees.

\begin{Theorem}\label{thm:lower}~Let $T$ be a tree of order $n \geq 4$ and $\Delta(T) \geq 3$. Then
\be\label{hc:lower} \hc(T) \geq (n-1)(n-1-\zeta(T))+\zeta'(T)-2\L_{W}(T). \ee
\end{Theorem}

The proof is similar to that of \cite[Theorem 4]{Bantva1} except only one change which is $\L(T)$ is replaced by $\L_{W}(T)$ but it should be noted that technically this lower bound is more useful than one given in \cite[Theorem 4]{Bantva1}. We support this fact by giving an example in section 4.

\begin{Theorem}\label{thm:main}~Let $T$ be a tree of order $n \geq 4$ and $\Delta(T) \geq 3$. Then
\be\label{hc:main} \hc(T) \geq (n-1)(n-1-\zeta(T))+\zeta'(T)-2\L_{W}(T) \ee
holds if and only if there exist an ordering $\{x_{0},x_{1},\ldots,x_{n-1}\}$ of the vertices of $T$, with $\L(x_{0})$ = 0 and $\L(x_{n-1})$ = 1 when $W(T) = \{w\}$ and $\L(x_{0})$ = $\L(x_{n-1})$ = 0 when $W(T) = \{w,w'\}$, such that for all $0 \leq i < j \leq n-1$,
\be\label{eq:dij} D(x_{i},x_{j}) \geq \displaystyle\sum_{t=i}^{j-1}\(\L(x_{t})+\L(x_{t+1})\)-(j-i)(n-1-\zeta(T))+(n-1). \ee
Moreover, under this condition the mapping $h$ defined by
\be\label{eq:f0} h(x_{0}) = 0 \ee
\be\label{eq:f1} h(x_{i+1}) = h(x_{i})+n-1-\zeta(T)-\L(x_{i})-\L(x_{i+1}).\ee
\end{Theorem}
\begin{proof}~\textsf{Necessity:} Suppose that \eqref{hc:main} holds. Let $h$ be an optimal hamiltonian coloring of $T$ then $h$ induces an ordering of vertices of $T$, say 0 = $h(x_{0}) < h(x_{1}) < ... < h(x_{n-1})$. The span of $h$ is $\span(h)$ = $\hc(T)$ = $(n-1)(n-1-\zeta(T))+\zeta'(T)-2\L_{W}(T)$. Note that this is possible if equality holds in \eqref{hc:dfn} together with (a) $\L(x_{0})$ = 0, $\L(x_{n-1})$ = 1 and $\phi(x_{i},x_{i+1})$ = $\delta(x_{i},x_{i+1})$ = 0 when $T$ has only one weight center, (b) $\L(x_{0})$ = $\L(x_{n-1})$ = 0 and $\phi(x_{i},x_{i+1})$ = 0 and $\delta(x_{i},x_{i+1})$ = 1 when $T$ has two adjacent weight centers. Note that this turn the definition of hamiltonian coloring $h(x_{i+1})-h(x_{i})$ = $n-1-D(x_{i},x_{i+1})$ to $h(x_{0})$ = 0 and $h(x_{i+1})$ = $h(x_{i})+n-1-\zeta(T)-\L(x_{i})-\L(x_{i+1})$ for $0 \leq i \leq n-2$. Moreover, for any two vertices $x_{i}$ and $x_{j}$ (without loss of generality, assume $j > i$), summing the latter equality for index $i$ to $j$, we have
\bean
h(x_{j})-h(x_{i}) & = & \displaystyle\sum_{t=i}^{j-1} [n-1-\zeta(T)-\L(x_{t})-\L(x_{t+1})].
\eean
Now $h$ is a hamiltonian coloring so that $h(x_{j})-h(x_{i}) \geq n-1-D(x_{i},x_{j})$ which turn the above equation in the following form.
\bean
D(x_{i},x_{j}) \geq \displaystyle\sum_{t=i}^{j-1}[\L(x_{t})+\L(x_{t+1})] - (j-i)(n-1-\zeta(T))+(n-1).
\eean
\textsf{Sufficiency:} Suppose that an ordering $\{x_{0},x_{1},\ldots,x_{n-1}\}$ of vertices of $T$ satisfies \eqref{eq:dij}, and $h$ is defined by \eqref{eq:f0} and \eqref{eq:f1} together with $\L(x_{0})$ = 0, $\L(x_{n-1})$ = 1 when $W(T) = \{w\}$ and $\L(x_{0})$ = $\L(x_{n-1})$ = 0 when $W(T) = \{w,w'\}$. Note that it is enough to prove that $h$ is a hamiltonian coloring with span equal to the right-hand side of \eqref{hc:main}. Let $x_{i}$ and $x_{j}$ ($j > i$) be two arbitrary vertices then by \eqref{eq:f1}, we have
\bean
h(x_{j})-h(x_{i}) & = & (j-i)(n-1-\zeta(T)) - \displaystyle\sum_{t=i}^{j-1} [\L(x_{t})+\L(x_{t+1})]
\eean
Now an ordering satisfies \eqref{eq:dij} which turn the above equation in the following form which shows that $h$ is a hamiltonian coloring.
\bean
h(x_{j})-h(x_{i}) \geq n - 1 - D(x_{i}, x_{j}).
\eean
The span of $h$ is given by
\bean
\span(h) & = & h(x_{n-1}) - h(x_{0}) \\
& = & \displaystyle\sum_{i=0}^{n-2} \(h(x_{i+1})-h(x_{i})\) \\
& = & (n-1)(n-1-\zeta(T)) - \displaystyle\sum_{i=0}^{n-2} (\L(x_{i})+\L(x_{i+1})) \\
& = & (n-1)(n-1-\zeta(T)) - 2\L_{W}(T) + \L(x_{0}) + \L(x_{n-1}) \\
& = & (n-1)(n-1-\zeta(T)) + \zeta'(T) - 2\L_{W}(T)
\eean
which complete the proof.
\end{proof}

Note that the following results gives sufficient conditions with optimal hamiltonian coloring for the equality in \eqref{hc:lower}. Note that these results are identical with \cite[Theorem 5]{Bantva1} and \cite[Corollary 1]{Bantva1} when $W(T)$ = $C(T)$ for a tree $T$. But it should be note that these results are more efficient than \cite[Theorem 5]{Bantva1} and \cite[Corollary 1]{Bantva1} (see the case of broom trees in section 4.2).

\begin{Theorem}\label{thm:s1}~Let $T$ be a tree of order $n \geq 4$ and $\Delta(T) \geq 3$. Then
\be\label{hc:s1} \hc(T) = (n-1)(n-1-\zeta(T))+\zeta'(T)-2\L_{W}(T) \ee
holds if there exist an ordering $\{x_{0},x_{1},\ldots,x_{n-1}\}$ of the vertices of $T$ such that for all $0 \leq i \leq n-2$,
\begin{enumerate}[\rm (a)]
\item $\L(x_{0}) + \L(x_{n-1})$ = 1 when $W(T) = \{w\}$ and $\L(x_{0}) + \L(x_{n-1})$ = 0 when $W(T) = \{w,w'\}$,
\item $x_{i}$ and $x_{i+1}$ are in different branches when $W(T) = \{w\}$ and in opposite branches when $W(T) = \{w,w'\}$,
\item $D(x_{i},x_{i+1}) \leq n/2$.
\end{enumerate}
Moreover, under these conditions the mapping $h$ defined by \eqref{eq:f0} and \eqref{eq:f1} is an optimal hamiltonian coloring of $T$.
\end{Theorem}
\begin{proof}~Suppose there exist an ordering $\{x_{0},x_{1},\ldots,x_{n-1}\}$ of vertices of $T$ such that (a), (b) and (c) holds and $h$ is defined by \eqref{eq:f0} and \eqref{eq:f1}. By Theorem \ref{thm:main}, it is enough to prove that an ordering $\{x_{0},x_{1},\ldots,x_{n-1}\}$ satisfies \eqref{eq:dij}. Let $x_{i}$ and $x_{j}$ be two arbitrary vertices, where $0 \leq i < j \leq n-1$. Without loss of generality, we assume that $j-i \geq 2$ and for simplicity let the right-hand side of \eqref{eq:dij} is $x_{i,j}$. Then, we obtain
\bean
x_{i,j} & = & \displaystyle\sum_{t=i}^{j-1} [\L(x_{t})+\L(x_{t+1})] - (j-i)(n-1-\zeta(T)) + (n-1) \\
& \leq & (j-i)(n/2-\zeta(T))-(j-i)(n-1-\zeta(T))+(n-1) \\
& = & (j-i)(1-n/2)+(n-1) \\
& \leq & 2(1-n/2)+(n-1) \\
& = & 1 \leq D(x_{i},x_{j})
\eean which completes the proof.
\end{proof}

For a tree $T$ of order $n$, if max$\{d(u,v) : u, v \in V(T), u \neq v\}$ $\leq$ $n/2$ then such a tree $T$ is called a tree with \emph{maximum distance bound $n/2$} or $DB(n/2)$ tree.

\begin{Corollary}\label{thm:s2}~Let $T$ be a DB($n/2$) tree of order $n \geq 4$ and $\Delta(T) \geq 3$. Then
\be\label{hc:s2} \hc(T) = (n-1)(n-1-\zeta(T))+\zeta'(T)-2\L_{W}(T) \ee
holds if and only if there exist an ordering $\{x_{0},x_{1},\ldots,x_{n-1}\}$ of the vertices of $T$ such that for all $0 \leq i \leq n-2$, \begin{enumerate}[\rm (a)]
\item $\L(x_{0}) + \L(x_{n-1})$ = 1 when $W(T) = \{w\}$ and $\L(x_{0}) + \L(x_{n-1})$ = 0 when $W(T) = \{w,w'\}$,
\item $x_{i}$ and $x_{i+1}$ are in different branches when $W(T) = \{w\}$ and in opposite branches when $W(T) = \{w,w'\}$,
\end{enumerate}
Moreover, under these conditions the mapping $h$ defined by \eqref{eq:f0} and \eqref{eq:f1} is an optimal hamiltonian coloring of $T$.
\end{Corollary}

\section{Theorem 1 vs \mbox{[3, Theorem 4]}}\label{sec:vs}

In this section, we determine the hamiltonian chromatic number of two families of trees namely $A_{d}$ and $B_{d}$, where $d=2k,2k+1$ and $k \geq 1$ be any integer. We show that in case of $W(T) = C(T)$, the lower bound for $\hc(T)$ given in Theorem \ref{hc:lower} and \cite[Theorem 4]{Bantva1} are identical but in case of $W(T) \neq C(T)$ then the lower bound for $\hc(T)$ given in  Theorem \ref{hc:lower} is better than \cite[Theorem 4]{Bantva1}. Note that in case of $A_{d}$, $W(A_{d}) = C(A_{d})$ but in case of $B_{d}$, $W(B_{d}) \neq C(B_{d})$, where $d=2k,2k+1$.

\subsection{Special Trees $A_{2k}$ and $A_{2k+1}$}\label{subsec:ak}

In \cite{Bala}, a special class of trees, namely $A_{2k}$ and $A_{2k+1}$, were introduced. The trees $A_{2k}$ with diameter $2k-1$ is defined as follows. (a) $A_{2}$ is an edge $K_{2}$, (b) $A_{2k} (k \geq 2)$ is obtained from $A_{2k-2}$ in the following manner: Add three pendant vertices (two vertical and one horizontal) to the left most and the right most pendant vertices of $A_{2k-2}$; next, add a pendant vertex (vertically) to every other pendant vertex of $A_{2k-2}$. In a manner analogous to the above, define the class of trees $A_{2k+1}$($k \geq 1$) with diameter $2k$ starting with $A_{3}$ which is $K_{1,4}$. The trees $A_{2k}$ and $A_{2k+1}$ has $2k^{2}$ and $2k(k+1)+1$ vertices respectively. Note that $|W(A_{2k})| = 2$ and $|W(A_{2k+1})| = 1$. Denote $W(A_{2k}) = \{w,w'\}$ and $W(A_{2k+1}) = \{w\}$. Moreover, it is clear from definition that $A_{2k}$ and $A_{2k+1}$ are $DB(n/2)$ trees.

\begin{Theorem}\label{thm:Ad}~Let $k \geq 1$ be any integer. Then
\begin{eqnarray}\label{hc:Ad}
\hc(A_{d}) := \left\{
  \begin{array}{ll}
    \frac{2}{3}(k-1)(6k^{3}+2k^{2}-4k-3), & \hbox{If }d = 2k, \\ [0.2cm]
    \frac{1}{3}k(k+1)(3k^{2}+k-1)+1, & \hbox{If }d = 2k+1.
  \end{array}
\right.
\end{eqnarray}
\end{Theorem}
\begin{proof}~The order $n$ and the total detour level $\L_{W}(A_{d})$ of $A_{d}$ $(d = 2k, 2k+1)$ are given by
\begin{eqnarray}\label{Ad:n}
n := \left\{
  \begin{array}{ll}
    2k^{2}, & \hbox{If }d = 2k, \\ [0.2cm]
    2k(k+1)+1, & \hbox{If }d = 2k+1.
  \end{array}
\right.
\end{eqnarray}

\begin{eqnarray}\label{Ad:l}
\L_{W}(A_{d}) := \left\{
  \begin{array}{ll}
    \frac{1}{3}k(k-1)(4k+1), & \hbox{If }d = 2k, \\ [0.2cm]
    \frac{2}{3}k(k+1)(2k+1), & \hbox{If }d = 2k+1.
  \end{array}
\right.
\end{eqnarray}
Substitute \eqref{Ad:n} and \eqref{Ad:l} into \eqref{hc:lower} we obtain the right-hand side of \eqref{hc:Ad} is a lower bound for $\hc(A_{d})$ $(d=2k,2k+1)$. Now we prove that this lower bound is tight and for this purpose it suffices to give a hamiltonian coloring whose span equal to the right-hand side of \eqref{hc:Ad}. Since $A_{d}$ $(d=2k,2k+1)$ are $DB(n/2)$ tree, by Corollary \ref{thm:s2} it is enough to give an ordering $\{x_{0},x_{1},\ldots,x_{n-1}\}$ of $V(A_{d})$ $(d=2k,2k+1)$ satisfying conditions (a) and (b) of Corollary \ref{thm:s2} and the hamiltonian coloring $h$ defined by \eqref{eq:f0}-\eqref{eq:f1} is an optimal hamiltonian coloring. For this we consider the following two cases for $A_{d}$.

\textsf{Case-1:} $d=2k$.~~In this case, note that $|W(A_{2k})| = 2$ and $A_{2k}$ has six branches. Denote the branches of $A_{2k}$ attached to $w$ by $T_{2},T_{4},T_{6}$ and $w'$ by $T_{1},T_{3},T_{5}$ such that $|T_{2}| > |T_{4}| = |T_{6}|$ and $|T_{1}| > |T_{3}| = |T_{5}|$. Define an ordering $\{x_0,x_1,\ldots,x_{n-1}\}$ of $V(A_{2k})$ as follows: Let $x_{0} = w$ and $x_{n-1} = w'$. For $1 \leq i \leq n-4k-1$, let

\begin{eqnarray*}
x_{i} := \left\{
  \begin{array}{ll}
    v \in V(T_{1}), & \hbox{If $i$ is odd}, \\ [0.2cm]
    v \in V(T_{2}), & \hbox{If $i$ is even}.
  \end{array}
\right.
\end{eqnarray*}

For $n-4k \leq i \leq n-2k-1$, let
\begin{eqnarray*}
x_{i} := \left\{
  \begin{array}{ll}
    v \in V(T_{3}), & \hbox{If $i$ is odd}, \\ [0.2cm]
    v \in V(T_{4}), & \hbox{If $i$ is even}.
  \end{array}
\right.
\end{eqnarray*}

For $n-2k \leq i \leq n-2$, let
\begin{eqnarray*}
x_{i} := \left\{
  \begin{array}{ll}
    v \in V(T_{5}), & \hbox{If $i$ is odd}, \\ [0.2cm]
    v \in V(T_{6}), & \hbox{If $i$ is even}.
  \end{array}
\right.
\end{eqnarray*}

In above ordering $v \in V(T_{i})$, where $i=1,2,\ldots,6$, we mean any $v$ from $V(T_{i})$ without repetition. Note that $\L(x_{0})+\L(x_{n-1}) = 0$ and $x_{i}$, $x_{i+1}$ ($0 \leq i \leq n-1$) are in opposite branches of $A_{2k}$.

\textsf{Case-2:} $d=2k+1$.~~In this case, note that $|W(A_{2k+1})| = 1$ and $A_{2k+1}$ has four branches. Denote the branches of $A_{2k+1}$ attached to $w$ by $T_{1},T_{2},T_{3},T_{4}$ such that $|T_{1}| = |T_{2}| > |T_{3}| = |T_{4}|$. Define an ordering $\{x_0,x_1,\ldots,x_{n-1}\}$ of $V(A_{2k+1})$ as follows: Let $x_{0} = w$ and $x_{n-1} = v \in V(T_{4})$ such that $v$ is adjacent to $w$. For $1 \leq i \leq n-2k-1$, let
\begin{eqnarray*}
x_{i} := \left\{
  \begin{array}{ll}
    v \in V(T_{1}), & \hbox{If $i$ is odd}, \\ [0.2cm]
    v \in V(T_{2}), & \hbox{If $i$ is even}.
  \end{array}
\right.
\end{eqnarray*}

For $n-2k \leq i \leq n-2$, let
\begin{eqnarray*}
x_{i} := \left\{
  \begin{array}{ll}
    v \in V(T_{3}), & \hbox{If $i$ is odd}, \\ [0.2cm]
    v \in V(T_{4}), & \hbox{If $i$ is even}.
  \end{array}
\right.
\end{eqnarray*}

Again in above ordering $v \in V(T_{i})$, where $i=1,2,3,4$, we mean any $v$ from $V(T_{i})$ without repetition. Note that $\L(x_{0})+\L(x_{n-1}) = 1$ and $x_{i}$, $x_{i+1}$ ($0 \leq i \leq n-1$) are in different branches of $A_{2k+1}$.
\end{proof}

\begin{Remark}~Note that in case of $A_{d}$ ($d=2k,2k+1$), $W(A_{d}) = C(A_{d})$ and hence both the lower bounds given in Theorem \ref{thm:lower} and \cite[Theorem 4]{Bantva1} are identical. Moreover, $\Delta(A_{2}) = 1$ but it is easy to verify that $\hc(A_{2}) = 0$ and the case of $A_{2}$ is also consistent with the formula in \eqref{hc:Ad}.
\end{Remark}

\begin{Example}An optimal hamiltonian coloring of $A_{2},A_{4},A_{6}$ and $A_{3},A_{5},A_{7}$ is shown in Fig.\ref{Fig:Ad}.
\end{Example}
\begin{figure}[h!]
  \centering
  \includegraphics[width=3in]{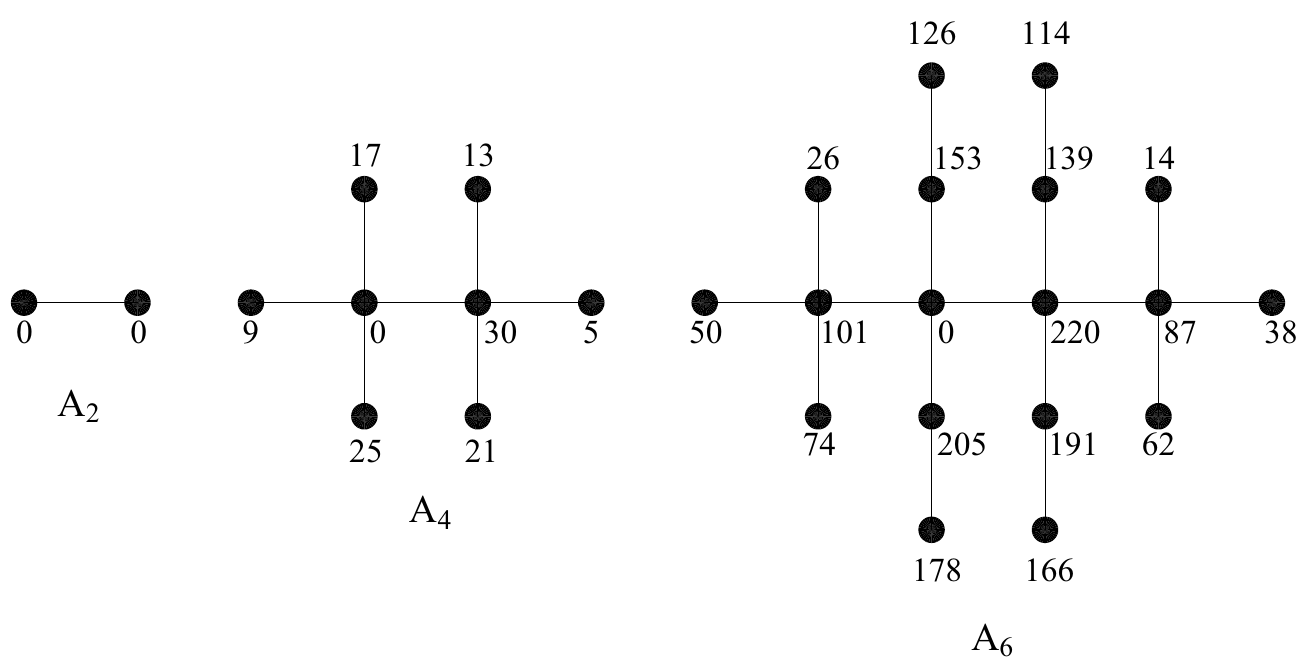}
  \includegraphics[width=3.8in]{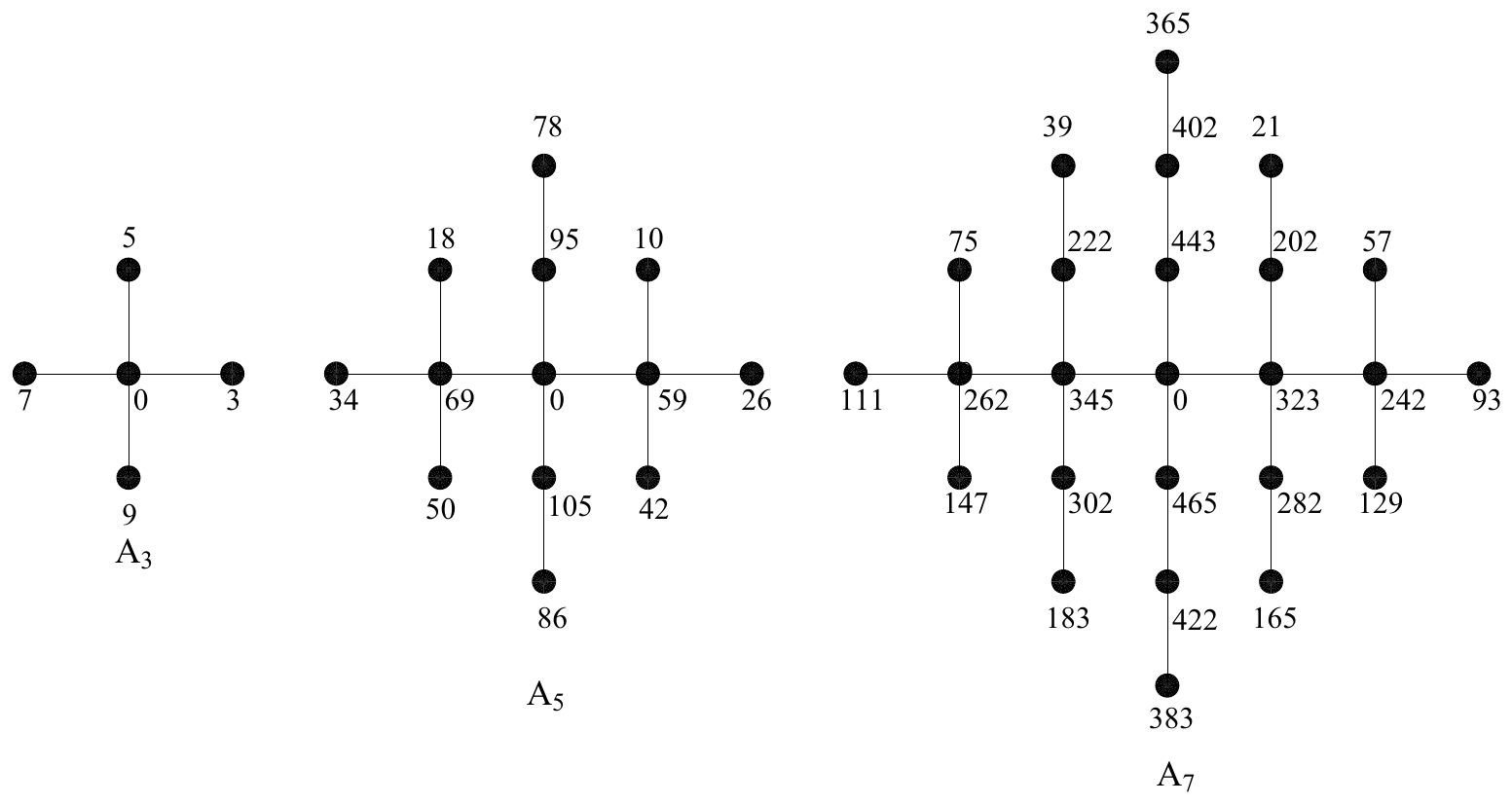}
  \caption{An optimal hamiltonian coloring of $A_{2},A_{4},A_{6}$ and $A_{3},A_{5},A_{7}$.}\label{Fig:Ad}
\end{figure}

\subsection{Broom Trees $B_{2k}$ and $B_{2k+1}$}\label{subsec:bk}

The broom tree $B_{n,d}$ consists of a path $P_{d}$, together with ($n-d$) end vertices all adjacent to the same end vertex of $P_{d}$. Let $k \geq 1$ be any integer. Define the broom trees $B_{2k}$ and $B_{2k+1}$ are broom tree $B_{k(2k+1),2k}$ and $B_{(k+1)(2k+1),2k+1}$, respectively. It is clear that $|W(B_{2k})| = |W(B_{2k+1})| = 1$. Denote $W(B_{2k}) = W(B_{2k+1}) = \{w\}$. Moreover, it is clear from definition that $B_{2k}$ and $B_{2k+1}$ are $DB(n/2)$ trees.

\begin{Theorem}\label{thm:broom}~Let $k \geq 1$ be any integer. Then
\begin{eqnarray}\label{hc:broom}
\hc(B_{d}) := \left\{
  \begin{array}{ll}
    2k(2k^{3}+2k^{2}-\frac{11}{2}k+1)+2, & \hbox{If }d = 2k, \\ [0.3cm]
    (2k+1)(2k^{3}+5k^{2}-2k-1)+2, & \hbox{If }d = 2k+1.
  \end{array}
\right.
\end{eqnarray}
\end{Theorem}
\begin{proof}~The order $n$ and the total detour level $\L_{W}(B_{d})$ of $B_{d}$ $(d = 2k, 2k+1)$ are given by
\begin{eqnarray}\label{broom:n}
n := \left\{
  \begin{array}{ll}
    k(2k+1), & \hbox{If }d = 2k, \\ [0.2cm]
    (k+1)(2k+1), & \hbox{If }d = 2k+1.
  \end{array}
\right.
\end{eqnarray}
\begin{eqnarray}\label{broom:l}
\L_{W}(B_{d}) := \left\{
  \begin{array}{ll}
    2k(2k-1), & \hbox{If }d = 2k, \\ [0.2cm]
    2k(2k+1), & \hbox{If }d = 2k+1.
  \end{array}
\right.
\end{eqnarray}
Substitute \eqref{broom:n} and \eqref{broom:l} into \eqref{hc:lower} we obtain the right-hand side of \eqref{hc:broom} is a lower bound for $\hc(B_{d})$ $(d=2k,2k+1)$. Now we prove that this lower bound is tight and for this purpose it is suffices to give a hamiltonian coloring whose span equal to the right-hand side of \eqref{hc:broom}. Since $B_{d}$ $(d=2k,2k+1)$ are $DB(n/2)$ tree, by Corollary \ref{thm:s2} it is enough to give an ordering $\{x_{0},x_{1},\ldots,x_{n-1}\}$ of $V(B_{d})$ $(d=2k,2k+1)$ satisfying conditions (a) and (b) of Corollary \ref{thm:s2} and the hamiltonian coloring $h$ defined by \eqref{eq:f0}-\eqref{eq:f1} is an optimal hamiltonian coloring. For this we consider the following two cases for $B_{d}$.

\textsf{Case-1:} $d=2k$.~~In this case, note that $B_{2k}$ has $k(2k-1)+1$ branches. Denote the branches of $B_{2k}$ by $T_{i}$, $i=1,2,\ldots,k(2k-1)+1$ such that $|T_{1}| > |T_{2}| = \ldots = |T_{k(2k-1)+1}|$. Let $S = \{v : v \in V(T_{1})\}$ and $S' =\{v : v \in V(T_{i}), i =2,3,\ldots,k(2k-1)+1\}$. Define an ordering $\{x_0,x_1,\ldots,x_n\}$ of $V(B_{2k})$ as follows: Let $x_{0} = w$ and for $1 \leq i \leq 4k-4$, let
\begin{eqnarray*}
x_{i} := \left\{
  \begin{array}{ll}
    v \in S, & \hbox{If $i$ is odd}, \\ [0.2cm]
    v \in S', & \hbox{If $i$ is even}.
  \end{array}
\right.
\end{eqnarray*}

For $4k-3 \leq i \leq n-1$, let $$x_i := v \in S'.$$

In above ordering $v \in X$, where $X=S,S'$ we mean any $v$ from $X$ without repetition. Note that $\L(x_{0})+\L(x_{n-1}) = 1$ and $x_{i}$, $x_{i+1}$ ($0 \leq i \leq n-1$) are in different branches of $B_{2k}$.

\textsf{Case-2:} $d=2k+1$.~~In this case, note that $B_{2k+1}$ has $k(2k+1)+1$ branches. Denote the branches of $B_{2k+1}$ by $T_{i}$, $i=1,2,\ldots,k(2k+1)+1$ such that $|T_{1}| > |T_{2}| = \ldots = |T_{k(2k+1)+1}|$. Let $S = \{v : v \in V(T_{1})\}$ and $S' =\{v : v \in V(T_{i}), i =2,3,\ldots,k(2k+1)+1\}$. Define an ordering $\{x_0,x_1,\ldots,x_n\}$ of $V(B_{2k+1})$ as follows: Let $x_{0} = w$ and for $1 \leq i \leq 4k$, let
\begin{eqnarray*}
x_{i} := \left\{
  \begin{array}{ll}
    v \in S, & \hbox{If $i$ is odd}, \\ [0.2cm]
    v \in S', & \hbox{If $i$ is even}.
  \end{array}
\right.
\end{eqnarray*}

For $4k+1 \leq i \leq n-1$, let $$x_i := v \in S'.$$

Again in above ordering $v \in X$, where $X=S,S'$ we mean any $v$ from $X$ without repetition. Note that $\L(x_{0})+\L(x_{n-1}) = 1$ and $x_{i}$, $x_{i+1}$ ($0 \leq i \leq n-1$) are in different branches of $B_{2k+1}$ which completes the proof.
\end{proof}

\begin{Example}~An optimal hamiltonian coloring of $B_{3},B_{5}$ and $B_{4},B_{6}$ is shown in Fig. \ref{Fig:Bd}.
\end{Example}
\begin{figure}[h!]
  \centering
  \includegraphics[width=4.5in]{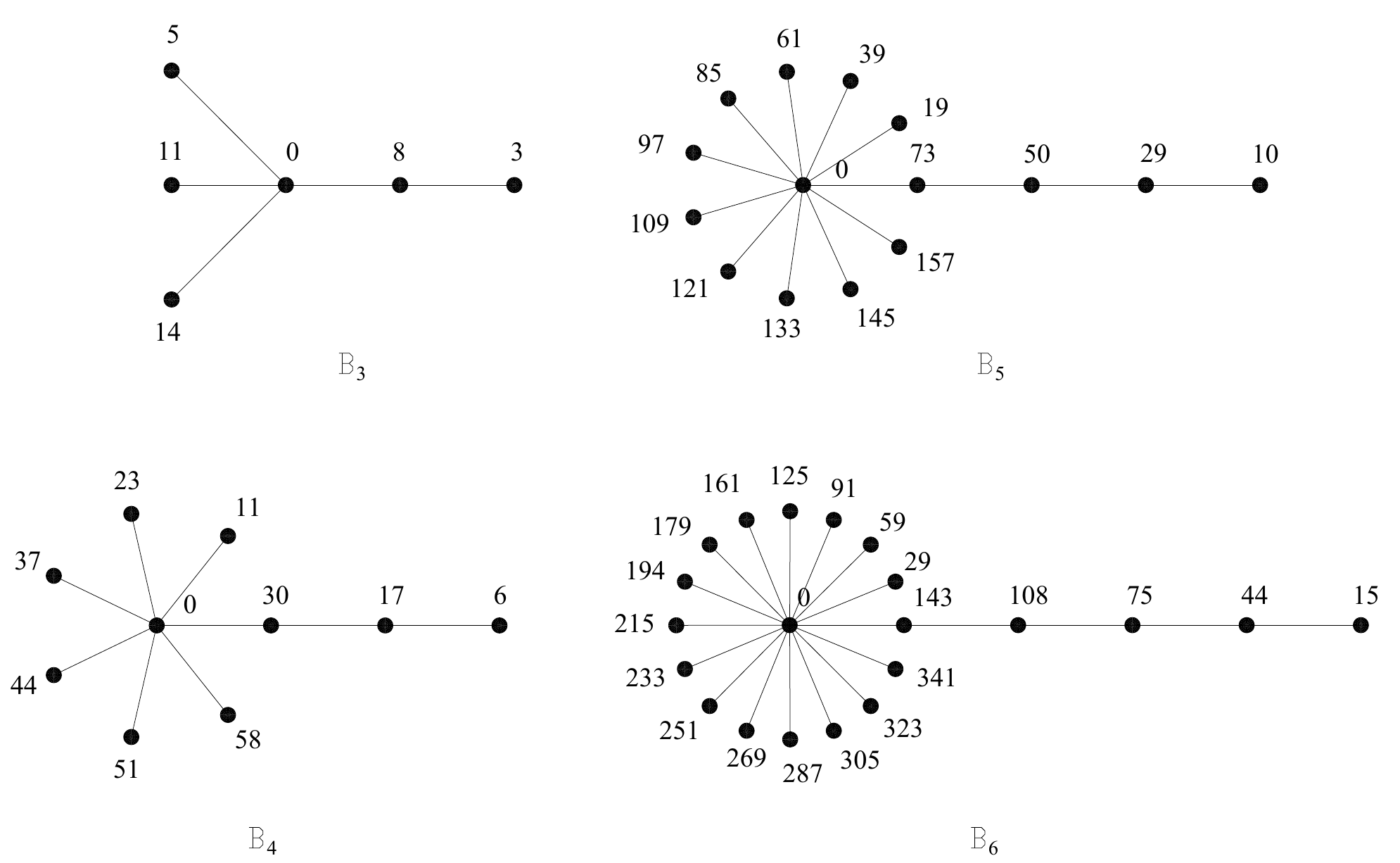}
  \caption{An optimal hamiltonian coloring of $B_{3},B_{5}$ and $B_{4},B_{6}$.}\label{Fig:Bd}
\end{figure}

We claim that the lower bound given in Theorem \ref{thm:lower} is better than the lower bound given in \cite[Theorem 4]{Bantva1} for broom trees $B_{d}$ ($d = 2k,2k+1$) except $B_{2}$. For broom tree $B_{2}$, both lower bounds are identical.  Denote the right-hand side of \eqref{hc:lower} and \cite[Equation-(2)]{Bantva1} by $lb_{W}(T)$ and $lb(T)$ respectively. That is, $lb_{W}(T) = (n-1)(n-1-\zeta(T))+\zeta'(T)-2\L_{W}(T)$ and $lb(T) = (n-1)(n-1-\ve(T))+\ve'(T)-2\L(T)$. Since $hc(B_{d}) = lb_{W}(B_{d})$ ($d=2k,2k+1 \geq 3$), it is enough to prove that $lb_{W}(B_{d}) - lb(B_{d}) > 0$ to justify our claim.

\begin{Theorem}~Let $k \geq 1$ be any integer. Then $lb_{W}(B_{d}) - lb(B_{d}) > 0$, where $d=2k,2k+1 \geq 3$.
\end{Theorem}
\begin{proof} We consider the following two cases.

\textsf{Case-1:} $d=2k$.~~In this case, $lb_{W}(B_{d}) - lb(B_{d}) = 4k(k-1)^{2} > 0$.

\textsf{Case-2:} $d=2k+1$.~~In this case, $lb_{W}(B_{d}) - lb(B_{d}) = 4k^{3}-2k^{2}-k+1 = k[(2k-1)^{2}+2(k-1)]+1 > 0$ which completes the proof.
\end{proof}

\section{Concluding remarks}\label{sec:con}

In \cite{Chartrand1}, Chartrand \emph{et al.} proved that for $n \geq 1$,
\begin{equation}\label{hc:K1n}
  \hc(K_{1,n-1}) = (n-2)^{2}.
\end{equation}
This result can be proved using Corollary \ref{thm:s2} as follows. The total detour level of $K_{1,n-1}$ is $\L(K_{1,n-1}) = n-1$ and $|W(K_{1,n-1})| = 1$. Substituting $\L(K_{1,n-1}) = n-1$ in \eqref{hc:lower}, we obtain that the right-hand side of \eqref{hc:K1n} is a lower bound for $\hc(K_{1,n-1})$ and it is easy to find a hamiltonian coloring whose span equal to this lower bound (refer \cite{Chartrand1}).

A tree is said to be a caterpillar $C$ if it consists of a path $v_{1},v_{2},\ldots,v_{m} (m \geq 3)$, called the spine of $C$, with some hanging edges known as legs, which are incident to the inner vertices $v_{2},v_{3},\ldots,v_{m-1}$. If $d(v_{i}) = d$ for $i=2,3,\ldots,m-1$, then denote the caterpillar by $C(m,d)$. In \cite{Shen}, Shen \emph{et al.} proved that for any positive integers $m \geq 3$ and $d \geq 3$,
\begin{eqnarray}\label{hc:Cmd}
\hc(C(m,d)) := \left\{
  \begin{array}{ll}
    \frac{2d-3}{2d-2}(n-2)^{2}+\frac{d-1}{2}, & \hbox{If $m$ is odd}, \\ [0.3cm]
    \frac{2d-3}{2d-2}(n-2)^{2}, & \hbox{If $m$ is even}.
  \end{array}
\right.
\end{eqnarray}
This result can also be proved using Corollary \ref{thm:s2} as follows. The order $n$ and total detour level $\L(C(m,d))$ are given by
\begin{eqnarray}\label{Cmd:n}
n := \left\{
  \begin{array}{ll}
    (2k-1)(d-1)+2, & \hbox{If $m = 2k+1$}, \\ [0.3cm]
    2k(d-1)-2(d-2), & \hbox{If $m = 2k$}.
  \end{array}
\right.
\end{eqnarray}
\begin{eqnarray}\label{Cmd:l}
\L(C(m,d)) := \left\{
  \begin{array}{ll}
    (k(k+1)-1)(d-1)+1, & \hbox{If $m = 2k+1$}, \\ [0.3cm]
    k(k-1)(d-1), & \hbox{If $m = 2k$}.
  \end{array}
\right.
\end{eqnarray}
Substituting \eqref{Cmd:n} and \eqref{Cmd:l} into \eqref{hc:lower} we obtain right-hand side of \eqref{hc:Cmd} is a lower bound for $\hc(C(m,d))$ and it is easy to find a hamiltonian coloring whose span equal to this lower bound (refer \cite{Shen}).

\end{document}